\documentclass{article}

\usepackage{arxiv}

\usepackage[utf8]{inputenc} 
\usepackage[T1]{fontenc}    
\usepackage{hyperref}       
\usepackage{url}            
\usepackage{booktabs}       
\usepackage{amsfonts}       
\usepackage{nicefrac}       
\usepackage{microtype}      
\usepackage{lipsum}		
\usepackage{graphicx}
\usepackage{doi}
\hypersetup{hidelinks}
\usepackage{graphicx}
\usepackage{xcolor}
\usepackage{makecell}
\usepackage{graphicx,times}
\usepackage{subfigure}         
\usepackage{diagbox}         
\usepackage{subfigure}         
\usepackage{graphicx}
\usepackage{amssymb}
\usepackage{lineno}
\usepackage{algorithm}
\usepackage{setspace}
\usepackage{array}
\usepackage{bigstrut}
\usepackage{multirow}
\usepackage{xcolor}
\usepackage{booktabs}
\usepackage{algorithmic}
\usepackage{indentfirst}
\usepackage{amsmath}
\usepackage{mathrsfs}
\usepackage{bm}
\usepackage{framed} 
\usepackage{adjustbox}
\usepackage[graphicx]{realboxes}
\usepackage{mathptmx}      
\usepackage{amsthm}

\newtheorem{theorem}{Theorem}[section]

\newtheorem{lemma}[theorem]{Lemma}

\title{A Modified Nonlinear Conjugate Gradient Algorithm for Functions with Non-Lipschitz continuous Gradient}

\author{ \href{https://orcid.org/0000-0003-0294-4294}{\includegraphics[scale=0.06]{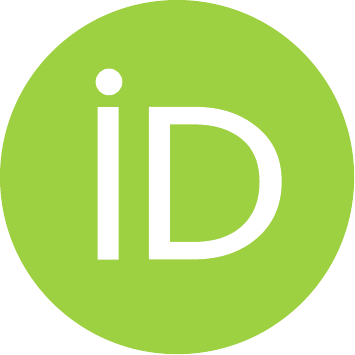}\hspace{1mm}Bingjie Li}\\
	Department of Statistics \& Data Science\\
	National University of Singapore, Singapore\\
	\texttt{bjlistat@nus.edu.sg} \\
	\And
	\href{https://orcid.org/0000-0002-5079-1029}{\includegraphics[scale=0.06]{orcid.pdf}\hspace{1mm}Tianhao Ni} \\
    School of Mathematical Science\\
    Zhejiang University, China\\ 
	\texttt{12035009@zju.edu.cn} \\
	\And
	\href{https://orcid.org/0000-0002-2485-7004}{\includegraphics[scale=0.06]{orcid.pdf}\hspace{1mm}Zhenyue Zhang} \\
	Nanjing Center for Applied Mathematics, China\\
    School of Mathematical Science,  Zhejiang University, China\\
	\texttt{zyzhang@zju.edu.cn} \\
}

\renewcommand{\shorttitle}{\textit{arXiv} Template}

\hypersetup{
pdftitle={A template for the arxiv style},
pdfsubject={q-bio.NC, q-bio.QM},
pdfauthor={David S.~Hippocampus, Elias D.~Striatum},
pdfkeywords={First keyword, Second keyword, More},
}
\begin{document}
\maketitle

\begin{abstract}
In this paper, we propose a modified nonlinear conjugate gradient (NCG) method for functions with a non-Lipschitz continuous gradient.  First, we present a new formula for the conjugate coefficient $\beta_k$ in NCG,  conducting a search direction that provides an adequate function decrease. We can derive that our NCG algorithm guarantees strongly convergent for continuous differential functions without Lipschitz continuous gradient. Second, we present a simple interpolation approach that could automatically achieve shrinkage, generating a step length satisfying the standard Wolfe conditions in each step.  Our framework considerably broadens the applicability of NCG and preserves the superior numerical performance of the PRP-type methods.
\end{abstract}

\keywords{nonlinear conjugate gradient \and line search \and strongly convergent \and non-Lipschitz gradient}

\section{Introduction}
\label{sec.1}
Consider the unconstrained optimization problem:
\begin{align}
    \min f(x) \quad  x\in \mathbb{R}^n,
\end{align}
where $f(x)$ is continuously differentiable.  The nonlinear conjugate gradient (NCG) method provides an iterative scheme for minimizing $f(x)$ via the two steps: starting at a point $x_0$ and setting $k=0$ and $d_0  = -\nabla f(x_0)$, the NCG updates the current point $x_k$ via a line search along the direction $d_k$ as that
\begin{align}\label{def:update x}
    x_{k+1} = x_k + \alpha_k d_k,
\end{align}
and then updates the direction $d_k$ in a linear combination of the previous direction $d_k$ and the current gradient $g_{k+1}=\nabla f(x_{k+1})$ for next search, 
\begin{align}\label{def:d}
    d_{k+1} = -g_{k+1} +\beta_k d_k.
\end{align}
In each iteration step,  the step length $\alpha_k$ and conjugate coefficient $\beta_k$ determine the convergence behavior of the NCG.  The sequence is called globally convergent if $\liminf_{k\to \infty} \|g_k\| = 0$, and called strongly convergent if $\lim_{k\to \infty} \|g_k\| = 0$.
\subsection{The step length}
\label{subsec.1.1}
Basically, given a search direction $d_k$, a step length $\alpha_k$ is chosen to yield a smaller value of the objective function $f$ at the updated point $x_{k+1}$. The ideal line search set the $\alpha_k$ that minimizes $f(x_k+\alpha d_k)$,
\[
    \alpha_k  = \arg\min_\alpha f(x_k+\alpha d_k).
\]
Since $f$ is nonlinear and not convex in many application, the ideal search may cost much. As a substitute, an inexact search approach is commonly used. In an inexact line search, the step length $\alpha_k$ is chosen to satisfy Armijo-Goldstein Condition \cite{A1966}
\begin{align}
    &f(x_k+\alpha_k d_k) 
    \leq f(x_k)+\rho\alpha_k \langle g_k,d_k\rangle \label{armijo1}\\ 
    &f(x_k+\alpha_k d_k) 
    \geq f(x_k)+(1-\rho)\alpha_k \langle g_k,d_k\rangle,  \label{armijo2}
\end{align}
or the standard Wolfe-Powell conditions \cite{W1969}
\begin{align}
    &f(x_k+\alpha_k d_k) 
    \leq f(x_k)+\rho\alpha_k \langle g_k,d_k\rangle \label{wolfe1}\\ 
    &\langle\nabla f(x_k+\alpha_k d_k),d_k\rangle 
    \geq \sigma \langle g_k,d_k \rangle,  \label{wolfe2}
\end{align}
or the strong Wolfe-Powell conditions\cite{W1971}
\begin{align}
    &f(x_k+\alpha_k d_k) 
    \leq f(x_k)+\rho\alpha_k \langle g_k,d_k\rangle \label{wolfe0}\\ 
    &|\langle\nabla f(x_k+\alpha_k d_k),d_k\rangle| 
    \leq -\sigma \langle g_k,d_k \rangle.\label{wolfe3}
\end{align}
Here, the parameters $\rho$ and $\sigma$ are positive and $\rho<\sigma<1$. Generally, $\alpha_k$ can be obtained via bisection \cite{MS1982} or interpolation \cite{SY2006}, or combination of the two approaches \cite{F2013}. 

\subsection{The search direction}
\label{subsec.1.2}
In nonlinear conjugate gradient methods, the behavior of $d_{k+1}$ is determined by $\beta_k$. There are many approaches in the literature to setting $\beta_k$.  Classical formulas for $\beta_k$ are called Fletcher-Reeves (FR) \cite{F1964}, Hestenes-Stiefel (HS) \cite{HS1952}, Polak-Ribiere-Polyak (PRP) \cite{P1969}. They are given by
\begin{align*}
    \beta_k^{\rm FR} = \frac{\|g_{k+1}\|_2^2}{\|g_k\|_2^2}, \quad
    \beta_k^{\rm HS} = \frac{\langle g_{k+1},  y_k\rangle}{\langle d_k,  y_k\rangle},\quad
    \beta_k^{\rm PRP} = \frac{\langle g_{k+1},y_k\rangle}{\|g_k\|_2^2},
\end{align*} 
where $y_k = g_{k+1}-g_k$. In practice, the PRP method outperforms others in many optimization problems because it can immediately recover after generating a tiny step. However, the PRP method only guarantees global convergence for strictly convex functions, limiting its applicability. To improve it, Gilbert and Nocedal \cite{GN1992} modified the PRP method by setting
\begin{align}
    \beta_k^{\rm PRP+} = \max\{\beta_k^{\rm PRP},0\}, 
\end{align}
and showed that this modification of the PRP method, called PRP+, is globally convergent if the search direction is sufficient descending and the step length satisfies the standard Wolfe conditions.  

In recent years, a variety of new nonlinear conjugate gradient methods have been proposed to find a search direction satisfying the descending condition
$
 \langle d_{k+1}, g_{k+1}\rangle <0
$
or the sufficient descending condition
$
 \langle d_{k+1}, g_{k+1}\rangle <-c\|g_{k+1}\|^2,
$
where $c$ is a positive number. In \cite{DY1999}, Dai and Yuan proposed a formula with
\begin{align}\label{DY}
    \beta_k^{\rm DY} &= \frac{\langle g_{k+1}, d_{k+1}\rangle }{\langle g_k,  d_k\rangle},
\end{align}
and it provides a descending direction. In \cite{HZ2005}, Hager and Zhang modified HS method to
\begin{align}
    \beta_k^{\rm HZ} = \max\{\beta_k^{\rm HS} - \frac{2\|y_{k}\|^2\langle g_{k+1}, d_k\rangle}{\langle d_k, y_{k} \rangle^2}, -\frac{1}{\|d_k\|\min\{\eta, \|g_k\|\}}\},\label{HZ}
\end{align}
where $\eta >0$ is a constant.  Similar modification on PRP method was proposed by Yuan \cite{Y2009}, that is,
\begin{align}
\beta_k^{\rm PRP-Y} = \max\{\beta_k^{\rm PRP} - \frac{\nu\|y_k\|^2}{\|g_k\|^4}
     \langle g_{k+1},d_k\rangle, 0\} \quad \nu>\frac{1}{4}. \label{PRP-Y}\end{align}
Both $\beta_k^{\rm HZ}$ and $\beta_k^{\rm MPRP}$ provide a sufficient descent direction. 

The convergence of the above nonlinear conjugate gradient methods requires the gradient $g(x)$ of the objective function $f(x)$ to be Lipschitz continuous. That is, there exists a constant $L>0$ such that 
\begin{align}
    \|g(x) - g(y) \| \leq L\|x-y\|, \quad \mbox{for all} \quad x,y \in \mathbb{R}^n.
\end{align}
This requirement of the Lipschitz continuous gradient limits the application of nonlinear conjugate gradient methods when faced with complicated practical problems.

\subsection{Our contribution}
\label{subsec.1.3}
In this paper, we propose a modified nonlinear conjugate gradient method, which does not require the gradient $g(x)$ to be Lipschitz continuous. The novelty of our approach comes from two aspects:
\begin{itemize}
    \item We propose a new formula for $\beta_k$, called MPRP, obtaining an adequate descending direction. The strong convergence of our approach is guaranteed  even though $f(x)$ is just a continuous differential function with a non-Lipschitz gradient.
    \item We suggest a more straightforward line search method for a step length that satisfies the standard Wolfe conditions in finite iterations. In practice, it works very well. The line search iteration terminates within one or two iterations generally in our experiments.
\end{itemize}

This paper is organized as follows. In Section \ref{sec.2}, we propose our new formula for $\beta_k$, and the line search approach is given in Section \ref{sec.3}. We discuss the convergence of our method in Section \ref{sec.4}. The numerical experiment is also given in Section \ref{sec.5}, to show the performance of our approach. At last, we end the article with a conclusion in Section \ref{sec.6}.

\section{The new formula for $\beta_k$}
\label{sec.2}
Besides the Lipschitz continuity of the gradient,  the convergence of the PRP-Y method \cite{Y2009} requires that the step length $\{\alpha_k\}$ have a positive lower bound. This requirement is difficult to guarantee in practice. To weaken the conditions required for the convergence of the PRP-Y method, we provide a formula for $\beta_k$ as
\begin{align}\label{MPRP}
    \beta_k = 
    \min\Big\{\frac{\langle g_{k+1},g_{k+1}-g_k
        -\frac{\nu\|g_{k+1}-g_k\|^2}{\|g_k\|^2}d_k\rangle_{+}}{\|g_k\|_2^2}, \frac{\kappa\|g_{k+1}\|_2}{\|d_k\|_2} \Big\}.
\end{align}
where $\nu>\frac{1}{4}$ as in (\ref{PRP-Y}) and $\kappa>0$. The modification can guarantee a stronger sufficient descent condition than that of the PRP-Y method. That is, 
\begin{lemma}\label{lma:sdc}
Let $\beta_k$ be defined by (\ref{MPRP}) and $\mu = \frac{4\nu-1}{4\nu(1+\kappa)}$. Then  
\begin{equation}\label{cond:sdc}
    \langle d_{k+1},g_{k+1} \rangle \leq -\mu \|d_{k+1}\|_2\|g_{k+1}\|_2.
\end{equation}
\end{lemma}

\begin{proof}
Let $\tilde\beta_k = \beta_k^{\rm PRP}$ for short, and let $\tilde d_{k+1} = -g_{k+1} +\tilde\beta_k d_k$. We rewrite 
\[
    \beta_k = \rho_k \tilde\beta_k,\quad
    d_{k+1} =\rho_k\tilde d_{k+1}+(\rho_k-1)g_{k+1}
\]
with a scale $\rho_k\in[0,1]$ since $\beta_k\leq \tilde\beta_k$. At first, we require the 
inequality 
\begin{align}\label{bound:beta_k}
    \tilde\beta_k\langle d_k,g_{k+1}\rangle 
    \leq \frac{1}{4\nu} \|g_{k+1}\|^2,
\end{align}
concluded by the definition 
$\tilde\beta_k
=\frac{\langle g_{k+1},g_{k+1}-g_k\rangle}      
{\|g_k\|_2^2}-\frac{\nu\|g_{k+1}-g_k\|^2}{\|g_k\|^4}
    \langle g_{k+1},d_k\rangle$.
It gives
\begin{align*}
    \tilde\beta_k\langle d_k,g_{k+1}\rangle 
    &= \frac{\langle g_{k+1},g_{k+1}-g_k\rangle}{\|g_k\|_2^2}\langle g_{k+1},d_k\rangle
        -\frac{\nu\|g_{k+1}-g_k\|^2}{\|g_k\|^4}
        \langle g_{k+1},d_k\rangle^2\\
    &= \Big\langle g_{k+1},\frac{\langle g_{k+1},d_k\rangle}{\|g_k\|_2^2}(g_{k+1}-g_k)\Big\rangle
        -\frac{\nu\langle g_{k+1},d_k\rangle^2}{\|g_k\|^4}
        \|g_{k+1}-g_k\|^2.
\end{align*}
Let $q_k = \frac{\langle g_{k+1},d_k\rangle}{\|g_k\|_2^2}(g_{k+1}-g_k)$ for simplicity. Then
\[
    \tilde\beta_k\langle d_k,g_{k+1}\rangle
    = \langle g_{k+1}, q_k\rangle-\nu\|q_k\|^2
    = \frac{\|g_{k+1}\|^2}{4\nu} - \|(2\sqrt{\nu})^{-1}g_{k+1} -\sqrt{\nu}q_k\|^2.
\]
Therefore, (\ref{bound:beta_k}) is true. Following it, we get that
$
    \langle\tilde d_{k+1},g_{k+1}\rangle 
    \leq (\frac{1}{4\nu}-1) \|g_{k+1}\|^2
$
and 
\begin{align*}
    \langle d_{k+1},g_{k+1} \rangle 
    =&\ \rho_k\langle \tilde d_{k+1},g_{k+1} \rangle +(\rho_k-1)\|g_{k+1}\|^2\\
    \leq&\ \big(\rho_k(\frac{1}{4\nu}-1)+(\rho_k-1)\big)\|g_{k+1}\|^2
    \leq \frac{1-4\nu}{4\nu}\|g_{k+1}\|^2.
\end{align*}
Here we have used $\rho_k\leq 1$. 
On the other hand, since $|\beta_k|\leq \frac{\kappa\|g_{k+1}\|}{\|d_k\|}$, we also have that  
\[
    \|d_{k+1}\|
    =\|-g_{k+1}+\beta_k d_k\|
    \leq\|g_{k+1}\|+|\beta_k|\|d_k\|
    \leq (1+\kappa)\|g_{k+1}\|.
\]
Therefore, (\ref{cond:sdc}) holds 
since $\|g_{k+1}\|_2^2\geq \frac{1}{1+\kappa}\|d_{k+1}\|_2\|g_{k+1}\|_2$ and $\frac{1-4\nu}{4\nu}<0$.
\end{proof}
We call a search direction $d_{k+1}$ an adequate descending direction if it satisfies (\ref{cond:sdc}). Obvious, the gradient itself satisfies (\ref{cond:sdc}) with $\mu = 1$. To the best of our knowledge, (\ref{MPRP}) is the first conjugate gradient method that provides an adequate descending direction.

\section{A simple interpolation line search approach}
\label{sec.3}
For a general continuously differentiable $f(x)$, the interpolation method does not guarantee capturing required $\alpha_k$ satisfying the standard Wolfe conditions since it asks for a three times continuously differentiable \cite{SY2006}. One can get $\alpha_k$ by the combination method \cite{F2013} that is more efficient than the bisection approach \cite{MS1982}. In \cite{F2013}, the bisection is combined with the interpolation in a bit complicated way for interval shrinking. Here we give a simpler approach for determining $\alpha_k$ satisfying the weak Wolfe-Powell conditions.

Theoretically, at a current point $x=x_k$ with the conjugate direction $d=d_k$, the required inexact line search $\alpha=\alpha_k$ satisfying the weak Wolfe-Powell conditions (\ref{wolfe1}-\ref{wolfe2}) can be chosen as
\begin{align}\label{alpha}
    \alpha^* = \sup\big\{\hat\alpha: \mbox{the Wolfe-Powell condition (\ref{wolfe1}) holds over $(0,\hat\alpha)$ }\big\}.
\end{align}
It exists, is positive, and satisfies (\ref{wolfe1}-\ref{wolfe2}). To verify this claim, let's consider the function
\[
    g(\alpha) = f(x)+\rho\alpha\langle\nabla f(x),d\rangle-f(x+\alpha d).
\]
Clearly, (\ref{wolfe1}) is equivalent to $g(\alpha)\geq 0$, and meanwhile, (\ref{wolfe2}) holds if $g'(\alpha)\leq0$. By the definition and the continuity of $f$, (\ref{wolfe1}) is true for $0<\alpha\leq\alpha^*$. The supremum in (\ref{alpha}) implies that $g(\alpha^*)=0$ and $g'(\alpha^*)\leq 0$. Hence, (\ref{wolfe2}) is also satisfied for $\alpha=\alpha^*$. Practically, there is a relative large sub-interval of $(0,\alpha^*]$ in which both (\ref{wolfe1}) and (\ref{wolfe2}) are true. For instance, if $\hat\alpha\in(0,\alpha^*]$ is the largest point such that $g(\alpha)$ is a local maximum, then $g'(\alpha)\leq 0$ in $[\hat\alpha,\alpha^*]$. Therefore, (\ref{wolfe1}-\ref{wolfe2}) hold for $\alpha\in[\hat\alpha,\alpha^*]$. 

An ideal choice of $\alpha$ is the minimizer $\alpha_{\min}$ of $f(x+\alpha d)$ over $(0, \alpha^*]$ since it decreases $f$ as small as possible, while both (\ref{wolfe1}) and (\ref{wolfe2}) are still satisfied.
In this subsection, we give a simple rule for pursuing $\alpha_{\min}$ via a quadratic interpolation to $f(x+\alpha d)$, assuming $f$ is continuously differentiable. It generates a nested and shrunk interval sequence containing the required $\alpha$. The pursuing terminates as soon as a point satisfying (\ref{wolfe1}-\ref{wolfe2}) is found.

Initially, we set $\alpha_0'=0$ that satisfies (\ref{wolfe1}) but (\ref{wolfe2}), and choose a relatively large $\alpha_0''>0$ that does not satisfy (\ref{wolfe1}). A simple choice of $\alpha_0''$ will be given later. Starting with $[\alpha_0', \alpha_0'']$, we generate a sequence of intervals $[\alpha_0', \alpha_0'']$ iteratively such that each $\alpha_\ell'$ satisfies (\ref{wolfe1}) but  $\alpha_\ell''$ does not, and meanwhile, $\alpha_\ell'$ doesn't satisfy (\ref{wolfe2}). That is, for $x_{\ell}' = x+\alpha_\ell'd$ and $x_{\ell}'' = x+\alpha_\ell''d$
\begin{align}\label{cond_end}
    f(x_{\ell}')\leq f(x)+\rho\alpha_\ell'\langle g,d\rangle,\
    f(x_{\ell}'')> f(x)+\rho\alpha_\ell''\langle g,d\rangle,\
    \langle\nabla f(x_{\ell}'),d\rangle<\sigma\langle g,d\rangle,
\end{align}
where $g = \nabla f(x)$. The third inequality above implies that $\langle\nabla f(x_{\ell}'),d\rangle<0$.  
Furthermore, by the first two inequalities in (\ref{cond_end}), we have that
\begin{align}\label{cond_alpha''}
    f(x_{\ell}'')> f(x_{\ell}')+\rho(\alpha_\ell''-\alpha_\ell')\langle g,d\rangle
    >f(x_{\ell}')+\frac{\rho}{\sigma}(\alpha_\ell''-\alpha_\ell')\langle\nabla f(x_{\ell}'),d\rangle.
\end{align}
In the current interval, we consider a quadratic function $q(\alpha)$ with interpolation conditions 
\[
    q(\alpha_\ell') = f(x_{\ell}'), \quad 
    q'(\alpha_\ell') = \langle\nabla f(x_{\ell}'),d\rangle, \quad 
    q(\alpha_\ell'') = f(x_{\ell}''),
\]
It can be represented as 
\begin{align*}
    q(\alpha) 
    = &\ f(x_{\ell}')+(\alpha-\alpha_\ell') \langle\nabla f(x_{\ell}'),d\rangle \\
    & +\big(f(x_{\ell}'')-f(x_{\ell}')
        -(\alpha_\ell''-\alpha_\ell')\langle\nabla f(x_{\ell}'),d\rangle\big)
    \frac{(\alpha-\alpha_\ell')^2}{(\alpha_\ell''-\alpha_\ell')^2}
\end{align*}
with the minimizer $c_\ell = \arg\min_{\alpha}q(\alpha)$ given by
\begin{align}\label{alpha_k}
    c_\ell = \alpha_\ell'
        +\frac{\alpha_\ell''-\alpha_\ell'}{2}
         \frac{-(\alpha_\ell''-\alpha_\ell')
                \langle\nabla f(x_{\ell}'),d\rangle}
            {f(x_{\ell}'')-f(x_{\ell}') -(\alpha_\ell''-\alpha_\ell')
             \langle\nabla f(x_{\ell}'),d\rangle}
    >\alpha_\ell'.
\end{align}
By the Mean-Value Theorem for derivatives and the second inequality in (\ref{cond_alpha''}), 
\begin{align}\label{sigma}
    0<(1-M_{\ell})^{-1}
    \leq&\ \Big(1-\frac{\langle\nabla f(\bar x_{\ell}),d\rangle}
                    {\langle\nabla f(x_{\ell}'),D_k\rangle}\Big)^{-1} \nonumber\\
    =&\ \frac{-(\alpha_\ell''-\alpha_\ell')
                \langle\nabla f(x_{\ell}'),D_k\rangle}
            {f(x_{\ell}'')-f(x_{\ell}') -(\alpha_\ell''-\alpha_\ell')
             \langle\nabla f(x_{\ell}'),D_k\rangle}
    <\frac{\sigma}{\sigma-\rho},
\end{align}
where $\bar x_{\ell} = x+\bar\alpha_\ell d$ with 
$\bar\alpha_\ell\in [\alpha_\ell', \alpha_\ell'']$ and 
\[
    M_{\ell} 
    = \min_{\alpha\in [\alpha_\ell', \alpha_\ell'']}
    \frac{\langle\nabla f(x+\alpha d),d\rangle}
         {\langle\nabla f(x_{\ell}'),d\rangle}
    \leq \frac{\langle\nabla f(\bar x_{\ell}),d\rangle}
              {\langle\nabla f(x_{\ell}'),d\rangle}
    <\frac{\rho}{\sigma}. 
\]
Hence,  if $0<2\rho<\sigma$, we have that
\begin{align}\label{c_ell bound}
    \alpha_\ell'<\alpha_\ell'+ \frac{1}{2(1-M_{\ell})}(\alpha_\ell''-\alpha_\ell')
    < c_\ell <\alpha_\ell'+\frac{\sigma}{2(\sigma-\rho)}(\alpha_\ell''-\alpha_\ell')
    <\alpha_\ell''.
\end{align}

We may shrink $[\alpha_\ell',\alpha_\ell'']$ to $[c_{\ell},\alpha_\ell'']$ or $[\alpha_\ell',c_{\ell}]$, if $\alpha=c_{\ell}$ satisfies (\ref{wolfe1}) or does not. However, if (\ref{wolfe1}) is satisfied, the interval length is 
$
    \alpha_\ell''-c_{\ell}
    \leq \frac{1-2M_{\ell}}{2-2M_{\ell}}(\alpha_\ell''-\alpha_\ell').
$
When $M_{\ell}<0$ and $|M_{\ell}|$ is large, $\frac{1-2M_{\ell}}{2-2M_{\ell}}\approx 1$. The interval shrinking is inefficient in this case. 
To avoid this phenomenon, we slightly modify $c_\ell$ as that with $\eta = \frac{\sigma}{2(\sigma - \rho)}$
\begin{align}\label{tilde_c}
    \tilde c_{\ell} 
    = \max\big\{ c_{\ell},\ \eta \alpha_{\ell}' +(1-\eta)\alpha_{\ell}'' \big\}
    \in(\alpha_\ell',\alpha_\ell'').
\end{align}
Since $\tilde c_{\ell} \geq\eta \alpha_{\ell}' +(1-\eta)\alpha_{\ell}''$ and $c_\ell<\alpha_\ell'+\eta(\alpha_\ell''-\alpha_\ell')$ by (\ref{c_ell bound}), we get
\[
    \alpha_{\ell}''- \tilde c_{\ell} \leq \eta(\alpha_\ell''-\alpha_\ell'),\quad
    \tilde c_{\ell}- \alpha_{\ell}'
      \leq \max\big\{\eta,1-\eta\big\}(\alpha_\ell''-\alpha_\ell')
      = \eta(\alpha_\ell''-\alpha_\ell').
\]
The last equality holds since $\eta>1/2$. Hence, if the Wolfe-Powell conditions (\ref{wolfe1}-\ref{wolfe2}) are satisfied for $\alpha = \tilde c_\ell$, we get the required $\alpha_k=\tilde c_\ell$. Otherwise, shrink $[\alpha_\ell',\alpha_\ell'']$ as 
\begin{align}\label{interval_alpha}
    [\alpha_{\ell+1}',\alpha_{\ell+1}'']
    =\left\{\begin{array}{ll}
         [\alpha_\ell', \tilde c_\ell], & \ \mbox{if (\ref{wolfe1}) does not hold for $\alpha = \tilde c_{\ell}$}; \\
         \mbox{$[\tilde c_\ell,\alpha_\ell'']$}, & \ \mbox{otherwise}.
    \end{array}
    \right.
\end{align}
The interval length is significantly decreased as $0<\alpha_{\ell+1}''-\alpha_{\ell+1}'\leq \eta(\alpha_{\ell}''-\alpha_{\ell}')$,
where $\eta<1$ since $2\rho <\sigma$. Hence, $\alpha_{\ell}''-\alpha_{\ell}'\to 0$ as $\ell\to \infty$. 

A good choice of $\alpha_0''$ helps to pursue the minimizer $\alpha_{\min}$. 
Motivated by the above analysis on the estimation of the shrinking rate $\eta_{\ell}$, we suggest the experiential setting 
\begin{align}\label{alpha_0''}
    \alpha_0'' = \min\big\{\alpha = 2^p\eta: \ 
    \mbox{(\ref{wolfe1}) is not satisfied for $\alpha=2^p\eta$ with integer $p\geq 0$} \big\}.
\end{align}
Algorithm \ref{alg:stepsize} gives the details of the procedure for determining an inexact line search $\alpha_k$, given $x_k$, $f_k$, $g_k$, the conjugate direction $d_k$.

\begin{algorithm}[t]
\setstretch{1.35}
\caption{An inexact line search satisfying the weak Wolfe-Powell conditions} \label{alg:stepsize}
\begin{algorithmic}[1]
    \REQUIRE point $x$, $f = f(x)$, $g=\nabla f(x)$, direction $d$, and parameters $\sigma$, $\rho$.
    \ENSURE $\alpha$ satisfying (\ref{wolfe1}-\ref{wolfe2}) within accuracy $\varepsilon$, $x := x+\alpha d$, $f(x)$, and $g=\nabla f(x)$.
    \STATE Set $\alpha' = 0$, $x' = x$, $f' = f$, $g' = g$, 
        $\nu = \rho\langle g,d\rangle$. Find the smallest integer $p\geq 1$ such that (\ref{wolfe1}) does not hold for $\alpha =\eta 2^p$, and set $\alpha''=\eta 2^p$.
    \STATE Repeat the following iteration until terminating. 
    \STATE \hspace{10pt} 
        Compute $c$ as (\ref{alpha_k}), $\tilde c$ as (\ref{tilde_c}), and 
        $\tilde f = f(\tilde x)$ at $\tilde x=x+\tilde cd$. 
    \STATE \hspace{10pt} 
        If $\tilde f > f+\tilde c\nu$, update $(\alpha'',f(x''))$ by $(\tilde c,f(\tilde x))$ and go to Step 3.
    \STATE \hspace{10pt} 
        Compute $\tilde g = \nabla f(\tilde x)$. 
        If $\langle\tilde g, d\rangle \geq \sigma\langle g,d \rangle$, set 
        $x =\tilde x$, $f = \tilde f$, $g = \tilde g$, and terminate.
    \STATE \hspace{10pt} 
        Otherwise, update $\alpha',f',g'$ by $\tilde c$, $\tilde f,\tilde g$, respectively.
    \STATE End iteration
\end{algorithmic}
\end{algorithm}
\section{Convergence of the Algorithm}
\label{sec.4}
Combining the formula (\ref{MPRP}) and line search Algorithm \ref{alg:stepsize}, we are able to provide our modified PRP-type (MPRP) nonlinear conjugate gradient method, as shown in Algorithm \ref{alg:ncg}. To show the convergence of the MPRP method, we first prove that the line search Algorithm \ref{alg:stepsize} will converge to a step length that satisfies the standard Wolfe condition.
\begin{algorithm}[t]
\setstretch{1.35}
\caption{The modified PRP-type (MPRP) nonlinear conjugated gradient method}
\label{alg:ncg}
\begin{algorithmic}[1]
    \REQUIRE initial point $x$, parameters $\epsilon$, $\rho$, $\sigma$, $\nu$, $\kappa$,  and $k_{\max}^{\rm NCG}$.
    \ENSURE an approximate solution $x_*$ of $\min_x f(x)$ with the given accuracy
    \STATE Compute $f = f(x)$, $g = \nabla f(x)$, and set $d = -g$.
    \STATE For $k=1,\cdots, k_{\max}^{\rm NCG}$,
    \STATE \hspace{10pt} 
        Update $(x,f,g)$ by \textbf{Algorithm \ref{alg:stepsize}} with searching direct $d$.
    \STATE \hspace{10pt} 
         If $\|g\|_{\infty}< \epsilon$, 
        then set $x^* = x$ and terminate the iteration. 
    \STATE \hspace{10pt} 
        Otherwise, compute $\beta$ by (\ref{MPRP}) and update $d := -g+\beta d$.
    \STATE End for
\end{algorithmic}
\end{algorithm}

\begin{lemma}
If $f$ is lower bounded and continuously differentiable, an $\alpha = \tilde c_{\ell^*}$ satisfying (\ref{wolfe1}-\ref{wolfe2}) can be obtained within a finite iterations of (\ref{interval_alpha}) if $0< 2\rho <\sigma<1$. 
\end{lemma}
\begin{proof}
If (\ref{wolfe1}-\ref{wolfe2}) do not hold for all $\tilde c_\ell$, the updating rule (\ref{interval_alpha}) yields a sequence of nested intervals $\{[\alpha_\ell',\ \alpha_\ell'']\}$. Since $0< 2\rho <\sigma<1$, the intervals tend to a single point $\alpha_*$ and both $\{x_\ell'\}$ and $\{x_\ell''\}$ tend to $x_* = x+\alpha_* d$. Hence, by (\ref{cond_alpha''}) and the Taylor extension of $f(x+\alpha d)$ at $\alpha=\alpha_*$, we get
\begin{align}\label{sigma_d}
    \langle\nabla f(x_*),d\rangle = 
    \lim_{\ell\to \infty} \frac{f(x_{\ell}'')-f(x_{\ell}')}{\alpha_\ell'' - \alpha_\ell'}
    \geq \rho \langle\nabla f(x),d\rangle
    > \sigma \langle\nabla f(x),d\rangle
\end{align}
since $\langle\nabla f(x),d\rangle<0$ and $\rho<\sigma$. However, by (\ref{cond_end}),
$\langle\nabla f(x_*),d\rangle\leq\sigma\langle\nabla f(x),d\rangle$, a contradiction with (\ref{sigma_d}).
\end{proof}
Because the search direction of MPRP is adequate descending, the proof of convergence of the algorithm is simple, similar to the proof of the steepest descent method. We have
\begin{theorem}\label{conv_CG}
Assume that $f$ is lower bounded and continuously derivative. If the inexact line search $\{\alpha_k\}$ satisfies the weak Wolfe-Powell condition (\ref{wolfe1}-\ref{wolfe2}) and 
\begin{equation}\label{cond:dg}
    \langle d_k,g_k \rangle \leq -\mu \|d_k\|_2\|g_k\|_2
\end{equation}
for a constant $\mu>0$, then the NCG converges: $\{f(x_k)\}$ is monotone decreasing and converges, and $\nabla f(x_k)\!\to\! 0$. 
\end{theorem}

\begin{proof}
We assume $g_k=\nabla f(x_k)\neq 0$ for each $k$  without loss of generalities, and let $s_k = \alpha_kd_k$. The condition (\ref{cond:dg}) becomes
$\langle g_k,s_{k}\rangle\leq-\mu\|g_k\|\|s_{k}\|\leq 0$. Hence, the Wolfe-Powell condition (\ref{wolfe1}) gives the monotone decreasing of $\{f(x_k)\}$,
\[
    f(x_{k+1})-f(x_k)\leq\rho\langle g_k,s_k\rangle
    \leq-\rho\mu\|g_k\|\|s_{k}\|\leq0,
\]
and $\{f(x_k)\}$ is convergent since $f$ itself is lower bounded. We also conclude from the convergence and the above inequality that  $\|g_k\|\|s_{\!k}\|\to 0$. 

We further show that $\|g_k\|\to 0$. Otherwise, there is a subsequence $\{\|g_{k_i}\|\}$ that has a positive lower bound. The lower bound implies that $\|s_{\!k_i}\|\to 0$ since   we also have $\|g_{k_i}\|\|s_{\!k_i}\|\to 0$. Note that $s_{k_i}=x_{k_i+1}-x_{k_i}$ is also the gap vector between $x_{k_i}$ and $x_{k_i+1}$, $f(x_{k_i})$ and $f(x_{k_i+1})$ can be represented each other in terms of $s_{k_i}$ via the Taylor extensions
\begin{align*}
    &f(x_{k_i+1}) 
    = f(x_{k_i}) + \langle g_{k_i},s_{k_i}\rangle + o\big(\|s_{k_i}\|\big),\\ 
    &f(x_{k_i}) 
    = f(x_{k_i+1}) - \langle g_{k_i+1},s_{k_i} \rangle + o\big(\|s_{k_i}\|\big).
\end{align*}
Clearly, the each other implies that $\langle g_{k_i},s_{k_i} \rangle - \langle g_{k_i+1},s_{k_i} \rangle  = o(\|s_{k_i}\|) $.

Turn back to the Wolfe-Powell condition (\ref{wolfe2}). Since it gives $\langle g_{k_i+1},s_{k_i}\rangle \geq \sigma \langle g_{k_i},s_{k_i}\rangle$,
\[
    o(\|s_{k_i}\|) 
    = \langle g_{k_i},s_{k_i} \rangle - \langle g_{k_i+1},s_{k_i} \rangle
    \leq (1-\sigma)\langle g_{k_i},s_{k_i}\rangle
    \leq -(1-\sigma)\mu\|g_{k_i}\|\|s_{k_i}\|.
\]
Here we have used the inequality $\langle g_{k_i},s_{k_i}\rangle \leq - \mu\|g_{k_i}\|\|s_{k_i}\|$ from the condition (\ref{cond:dg}) and $\sigma < 1$. Hence, 
$(1-\sigma)\mu \|g_{k_i}\|\leq -\frac{o(\|s_{k_i}\|)}{\|s_{k_i}\|}$, and
\[
    0\leq (1-\sigma)\mu\liminf \|g_{k_i}\|
    \leq 0.
\]
It implies $(1-\sigma)\mu\leq 0$, a contradiction, since $\liminf \|g_{k_i}\|>0$ by assumption.
\end{proof}
\section{Numerical Experiments}
\label{sec.5}
In this section, we show the performance of our MPRP in Algorithm \ref{alg:ncg}, compared with three other PRP-type NCG methods:  the classical PRP method \cite{P1969}, the PRP+ method \cite{GN1992} and the PRP-Y method \cite{Y2009}. 
The numerical experiments are divided into 3 parts. In the first part, we test these PRP-type methods on $84$ unconstrained optimization problems with Lipschitz continuous gradient from \cite{A2008}. The second part aims to show the enhancements brought by our line search method, compared with the bisection line search method. In the third part, we adopt our algorithm on a regression problem with an objective function whose gradient is non-Lipschitz continuous. The following parameters were adopted in our implementation
\begin{align*}
    \nu = 0.8,\ \kappa = 10,\ \rho = 0.1,\ \sigma = 0.4. 
\end{align*}
We set the termination criterion as $\|\nabla f(x)\|_{\infty}\leq 10^{-5}$ and the maximum number of iterations as $20000$. All compared algorithms are executed on the Windows system in a PC with Intel Core i5-8250U CPU@1.80GHz and 8GB RAM. 

We adopt the commonly used performance profile of Dolan and Mor\'{e} proposed in \cite{DM2002}, to display the performance of compared NCG methods in terms of CPU time and the number of iterations.   Take the CPU time as an example, let $S$ and $P$ be the set of solvers and problems, respectively, and denote $n_s = |S|$ and $n_p = |P|$. For each solver $s$ and problem $p$, let $t_{p,s}$ be the computing time requiring by solver $s$ to solve problem $p$. For each problem $p$, define the performance ratio as $r_{p,s} = \frac{t_{p,s}}{\min\{t_{p,s},s\in S\}}$ and the ratio $r_{p,s}\ge 1$ obviously for all $p$ and $s$. If a solver fails to  solve a problem, the ratio $r_{p,s}$ is set to a large enough positive number $M$ that larger than $r_{p,s}$ of problem $p$ that can be solved by solver $s$. Finally, the performance profile is defined by 
$$\rho(\tau) = \frac{1}{n_p}|\{p\in P:r_{p,s}\leq \tau\}|$$. The performance profile of the number of iterations is similar.

\subsection{The performance of compared NCG methods on tested functions}
\label{subsec.5.1}
This subsection shows the performance of 4 tested NCG algorithms on 84 unconstrained optimization problems drawn from \cite{A2008}.  The performance profile of CPU time and number of iterations is illustrated in Figure \ref{Comparison four NCG}. 
\begin{figure}[h]
\centering
\includegraphics[scale = 0.6]{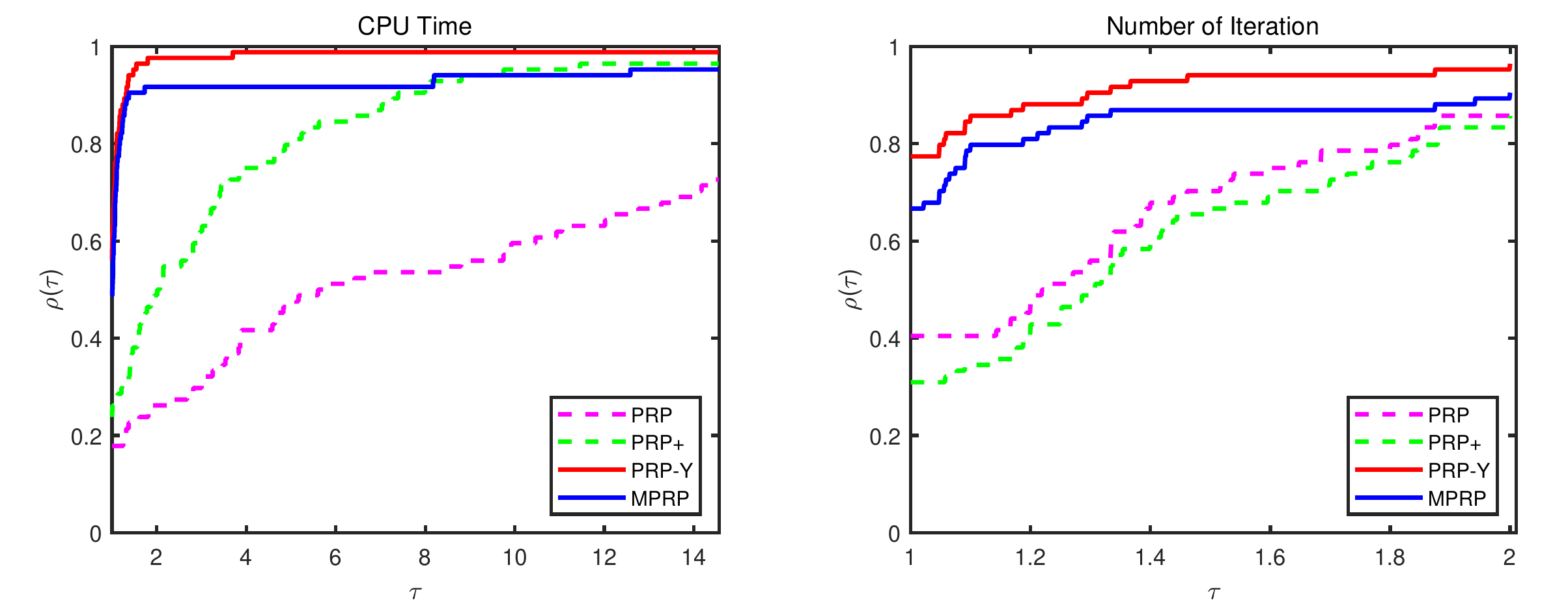}
\caption{Performance profile of PRP,PRP+,PRP-Y and MPRP with CPU Time and Number of iteration}
\label{Comparison four NCG}
\end{figure}

Among the four algorithms, MPRP and PRP-Y have significant advantages over the PRP and PRP+ method in the CPU Time and the number of iterations. However, PRP-Y performs slightly better than our MPRP. This is due to two reasons. For one thing, MPRP controls the angle between the search direction and the negative gradient, weakening the effect of the conjugate direction. Hence, the zig-zag phenomenon may occur in a small part of optimization problems.  For another thing, 
MPRP requires an additional computation of $\|d_k\|$ at each step, which also increases the CPU time slightly.

\subsection{Line searches via bisection and interpolation}
\label{subsec.5.2}
In this subsection, we test the performance of our interpolation line search in Algorithm \ref{alg:stepsize},  compared with the classical bisection line search \cite{MS1982}. The two line search approaches are used for the 4 tested NCG methods, on the 84 unconstrained optimization problems in \cite{A2008}. The Comparison is shown on CPU time and the number of iterations in Figure \ref{Comparison of line search}. 

\begin{figure}[t]
\centering
\includegraphics[scale = 0.6]{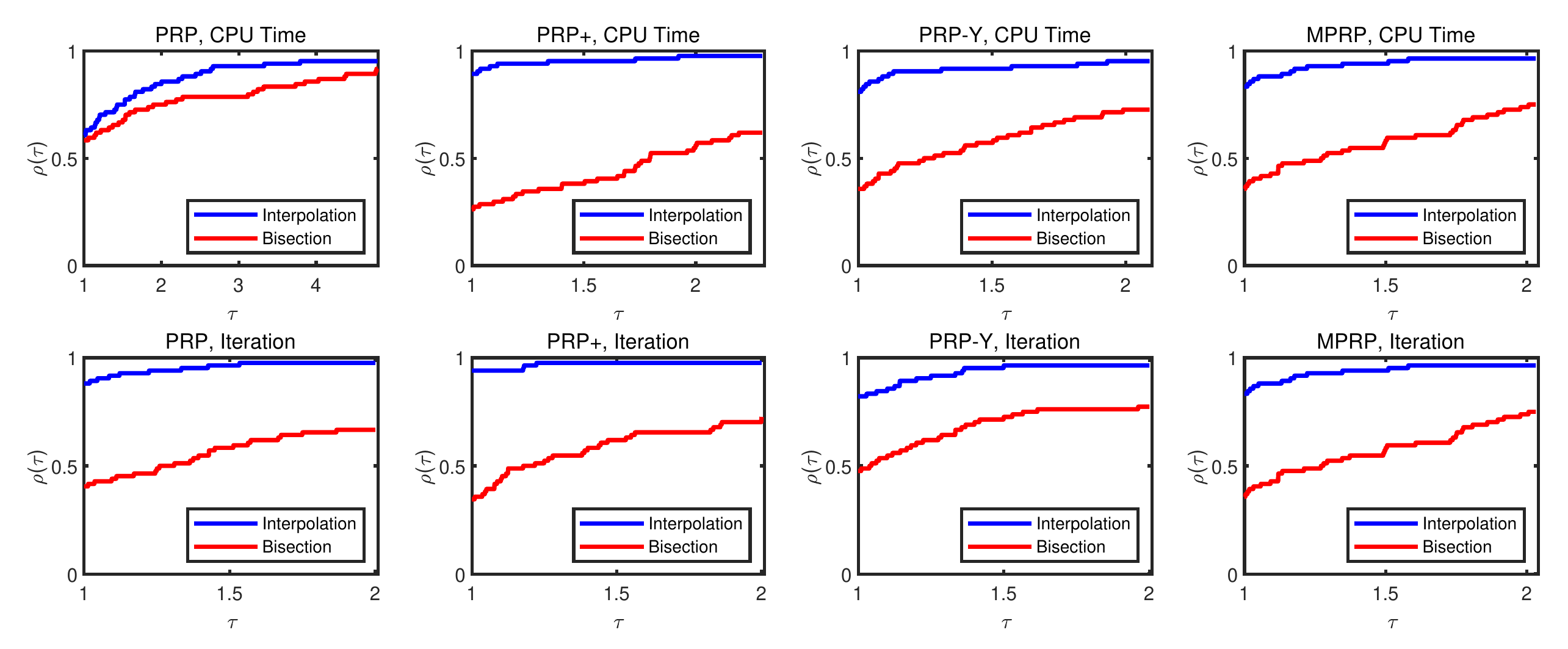}
\caption{Performance profile of different line search methods on PRP,PRP+,PRP-Y and MPRP with CPU Time and Number of iterations}
\label{Comparison of line search}
\end{figure}

As shown in Figure \ref{Comparison of line search}, the interpolation line search performs much better than the bisection line search in most problems regardless of the NCG methods we adopt. It has been shown that our interpolation line search method saves a significant amount of computation both in terms of the number of iterations and the CPU time.

\subsection{Performance of NCG methods on a function with non-Lipschitz continuous gradient}
\label{subsec.5.3}
In this subsection, we consider a linear regression model with a regular term as follows:
\begin{align}\label{regression}
    \min_x \frac{1}{2}\|Ax-b\|_2^2 + \frac{\lambda}{2}\|x\|_p^p,
\end{align}
where $\|x\|_p^p = |x_1|^p+...+|x_n|^p$. The model (\ref{regression}) becomes the lasso regression model or the ridge regression model, if one select $p = 1$, or $p = 2$, respectively. Here we set $p = 1.5$ so that $f(x) = \frac{1}{2}\|Ax-b\|_2^2 + \frac{\lambda}{2}\|x\|_p^p$ is continuous differential with a non-Lipschitz gradient, and test the compared NCG methods on it. 

In our experiment, we set $A$ as a random matrix in $\mathbb{R}^{10\times 50}$ with entries drawn from the uniform distribution in $[0,1]$, and $b = Au$ with $u$ is sparse with $10\%$ non-zero entries drawn from the standard normal distribution. We also set $\lambda = 0.01$.  We test the four algorithms 10 times, each with different random seeds. 
\begin{table}[h]
  \centering
  \caption{The performance of four NCG methods on (\ref{regression}) with $p = 1.5$.}
    \begin{tabular}{ccccc}
    \hline
    Method & PRP   & PRP+  & PRP-Y & MPRP \bigstrut\\
    \hline
    CPU Time  & 1.13  & 1.01  & 0.81  & 0.81 \bigstrut[t]\\
    Iterations & 504.5 & 447.9 & 422.9 & 419.4 \bigstrut[b]\\
    \hline
    \end{tabular}%
  \label{tab:regre}%
\end{table}%

Table \ref{tab:regre} shows the average CPU time and the average number of iterations for the 10 runs. It can be seen that the performance of PRP-Y and MPRP is slightly better than that of the other two algorithms. Surprisingly, PRP, PRP+ and PRP-Y all converge on this problem, even though they do not theoretically have a guarantee of convergence. One conjecture is that these algorithms skip the non-Lipschitz region of the gradient and converge to a stationary point with a neighbor where the gradient is Lipschitz continuous.

\section{Conclusion}
\label{sec.6}
This paper proposed a modified nonlinear conjugate gradient method for continuous differential function. The strong convergence is guaranteed without the condition of the Lipschitz continuous gradient. Furthermore, a simpler but more efficient interpolation Wolfe line search method is also introduced. The numerical results demonstrate the feasibility of the new NCG method and the new line search method. However, although in theory our algorithm gains a greater range of applicability, this advantage does not manifest itself numerically. The reasons behind this are worth further investigation.

\section*{Acknowledgments}
The work was supported by NSFC project 11971430.

\bibliographystyle{unsrt}

\end{document}